\documentclass[12pt]{amsart}
\usepackage{amssymb,amscd}

\usepackage{graphicx}
\usepackage[shortlabels]{enumitem}

\usepackage{braket}
\usepackage{tikz}

\numberwithin{equation}{section}        % to get #'s with section # (good numbering!)

\newcommand{\bbC}{{\mathbb C}}           % complex number field
\newcommand{\bbD}{{\mathbb D}}           % unit disk in \bbC
\newcommand{\bbR}{{\mathbb R}}           % real number field
\newcommand{\Co}{{\bbC}_*}			     % punctured plane
\newcommand{\zbar}{\bar{z}}

\swapnumbers                        % to put numbers in front of....
\theoremstyle{plain}
\newtheorem*{thm*}{Theorem}         % for Theorems, Lemmas,
\newtheorem*{lma*}{Lemma}           % etc with NO numbering
\newtheorem*{cor*}{Corollary}

\newtheorem{lemma}[equation]{Lemma}

\theoremstyle{remark}
\newtheorem{rmks}[equation]{Remarks}

\def\IR{{\mathbb R}}

\def\IC{{\mathbb C}}
\def\ID{{\mathbb D}}

\begin{document} %>>--BEGIN--<<%%
%%===============%=============%%
\title{Subharmonic Functions, Conformal Metrics, and CAT(0)}
\date{\today}

%--------------------%
% author information %
%--------------------%

\author{David A.\ Herron}
\address{Department of Mathematical Sciences, University of Cincinnati, OH 45221-0025, USA}
\email{David.Herron@UC.edu}

\author{Gaven J.\ Martin}
\address{Institute for Advanced Study, Massey University, New Zealand}
\email{G.J.Martin@Massey.ac.NZ}

%----------%
% AMS info %
%----------%
%  I did check these!

%----------%
% Abstract %
%----------%
\begin{abstract}
We present an analytical proof that certain natural metric planar universal covers are Hadamard metric spaces.  In particular if $\rho=\varphi\circ u$ where $u$ is locally Lipschitz and subharmonic in $\Omega$, $\varphi$ is positive and increasing on an interval containing $u(\Omega)$ with $\log\varphi$ convex, and if the metric space $(\Omega,\rho(z)|dz|)$ is complete,  then it has universal cover
$(\tilde{\Omega},\tilde{d})$ which is a Hadamard space for which geodesics have Lipschitz continuous first derivatives.

\end{abstract}

%------------%
% dedication %
%------------%
\dedicatory{Dedicated to Pekka Koskela on the occasion of his $60^{\rm th}$ birthday.}

\thanks{This research was initiated at the Bedlewo Mathematical Research and Conference Center where both authors were supported by the Polish Academy of Sciences.  We thank them for their hospitality. GJM is supported in part by a grant from the NZ Marsden Fund.}

\maketitle

\section{Introduction}

In an earlier paper \cite{Kos98}, Pekka Koskela wrote about old and new results concerning the quasihyperbolic metric - a metric introduced in \cite{GP76}.  Since then this metric has been a key player in much of Pekka's work along with that of many other authors, either as a fundamental geometric tool or as the focus of research.  We anticipate that Pekka will find the following results of interest as they include quasihyperbolic geometry as a special case.

Throughout this article $\Omega$ denotes a plane domain in the complex number field $\IC$; so, $\Omega\subset\bbC$ is open and connected.  The \emph{Gaussian curvature} of a (sufficiently) smooth conformal metric $\rho\,ds$ (see \S\ref{s:cfml metrics}) is given by
\begin{equation}\label{E:cvtr}
  {\mathbf K}_\rho:=-\rho^{-2}\Delta\log\rho\,.
\end{equation}

Classical results---for instance, see \cite[Theorem~1A.6, p.173 and Chapter~II.4, Theorem~4.1, p.193]{BH99}---reveal that if $\rho$ is $\mathcal{C}^3$ smooth, $\log\rho$ is subharmonic in $\Omega$, and the length distance $d$ induced by $\rho\,ds$ is complete, then $(\Omega,d)$ is a metric space of non-positive curvature (equivalently, the universal metric cover $(\tilde{\Omega},\tilde{d})$ of $(\Omega,d)$ is Hadamard).

A natural question is whether or not we can relax the above smoothness hypothesis, and we answer this as follows.

%%+++++++++++++++++++++++++++%%
\begin{thm*}  \label{TT:main}%%
%%+++++++++++++++++++++++++++%%
Let $\rho\,ds$ be a conformal metric on a plane domain $\Omega$ with a complete induced length distance $d$,  see (\ref{metric}).  Suppose $\rho=\varphi\circ u$ where $u$ is continuous and subharmonic in $\Omega$, $\varphi$ is positive and increasing on an interval containing $u(\Omega)$, and $\log\varphi$ is convex.  Then the metric universal cover
$(\tilde{\Omega},\tilde{d})$ is a Hadamard space.
\end{thm*}                   %%
%%+++++++++++++++++++++++++++%%

Recall that a \emph{Hadamard space} is a complete CAT(0) metric space; see \S\ref{s:CAT0}.

\smallskip

When $\tilde{\Omega}\stackrel{\Phi}{\to}\Omega$ is a universal cover of $\Omega$, and $(\Omega,d)$ is a length space, there is a unique length distance $\tilde{d}$ on $\tilde{\Omega}$ such that $(\tilde{\Omega},\tilde{d})\stackrel{\Phi}{\to}(\Omega,d)$ is a local isometry and $\tilde{d}$ is given by
\begin{equation}\label{metric}
  \tilde{d}(a,b):=\inf\bigl\{\ell_\rho(\Phi\circ\gamma) \bigm| \gamma\;\text{a path in $\tilde{\Omega}$ with endpoints $a, b$}\bigl\}\,;
\end{equation}
see \cite[Prop.~3.25, p.42]{BH99} or \cite[p.80]{BBI01}.  We call $(\tilde{\Omega},\tilde{d})$ the \emph{metric universal cover}  of $(\Omega,d)$.

\smallskip

There are some immediate consequences of the above Theorem.

%%~~~~~~~~~~~~~~~~~~~~~~~~~~%%
\begin{cor*}  \label{CC:sc}
Suppose $d$ is the length distance induced by a conformal metric $\rho\,ds$ on $\Omega$ satisfying the above hypotheses.
\begin{enumerate}[\rm(a), wide, labelwidth=!, labelindent=0pt]
%  \item[ ]
  \item  If $\Omega$ is simply connected, then $(\Omega,d)$ is Hadamard and for each pair of points $a,b\in\Omega$ there is a unique $d$-geodesic with endpoints $a,b$.
  \item  For each pair of points $a,b$ in any $\Omega$, each homotopy class of paths in $\Omega$ with endpoints $a,b$ contains a unique $d$-geodesic.
  \item Should $\rho:\Omega\to\IR_+$ be locally Lipschitz,  these geodesics have Lipschitz continuous first derivatives.
\end{enumerate}
\end{cor*}                  %%
%%~~~~~~~~~~~~~~~~~~~~~~~~~~%%

\noindent{\bf Examples.} When $\Omega\subsetneq\bbC$, $\rho_\alpha:=\delta^\alpha$, where $\alpha\in\bbR$ and $\delta(z)=\delta_\Omega(z):={\rm dist}\,(z,\partial \Omega )$ is the Euclidean distance from $z$ to the boundary of $\Omega$, defines a conformal metric $\rho_\alpha\,ds$ on $\Omega$.  Since $-\log\delta$ is subharmonic in $\Omega$,\footnote{%
For fixed $\zeta\in\partial\Omega$, $z\mapsto-\log|z-\zeta|$ is harmonic in $\bbC\setminus\{\zeta\}\supset\Omega$, so it has the mean value property in $\Omega$, whence $-\log\delta$ does too.}
$\log\rho_\alpha$ is subharmonic when $\alpha<0$; it also induces a complete distance for $\alpha\le-1$.  Thus $\set{\rho_\alpha\,ds | \alpha\le-1}$ is a class of metrics to which we can apply the above Theorem and Corollary.

For $\alpha:=-1$ we obtain the \emph{quasihyperbolic metric} $\delta^{-1}ds=\delta_\Omega^{-1}ds$, and this special case of our Theorem gives \cite[Theorem~A]{Her21a}.

Furthermore, if $\gamma\geq 1$, then locally
\[ |\delta^\gamma_\Omega(z)-\delta^\gamma_\Omega(w)| \leq 2\alpha \delta^{\gamma-1}_\Omega(z)\; |z-w| \]
by the triangle inequality.  Hence all these metrics have $C^{1,1}$ geodesics.

\section{Preliminaries}
%%=========================================%%
%
%---------------------------------------------------%
\subsection{General Information}
%---------------------------------------------------%

We view the Euclidean plane as the complex number field $\IC$.  Everywhere $\Omega$ is a {\em plane domain}.
The open unit disk is $\ID:=\{z\in\IC:|z|<1\}$ and $\IC_*:=\IC\setminus\{0\}$.  The quantity $\delta(z)=\delta_\Omega(z):={\rm dist}\,(z,\partial\Omega)$ is the Euclidean distance from $z\in\IC$ to the boundary of $\Omega$.
%%% and $1/\delta$ is the scaling factor (that is the metric-density) for the so-called \emph{quasihyperbolic metric} $\delta^{-1}ds$ on $\Omega\subset\IC$; see \S\ref{s:cfml metrics}.

%----------------------------------------------------------%
\subsection{Conformal Metrics}      \label{s:cfml metrics} %
%----------------------------------------------------------%
Each positive continuous function $\Omega\xrightarrow{\rho}(0,+\infty)$ induces a length distance $d_\rho$ on $\Omega$ defined by
\[
  d_\rho(a,b):=\inf_{\gamma } \ell_\rho(\gamma) \quad\text{where}\quad \ell_\rho(\gamma):= \int_\gamma \rho\, ds
\]
and where the infimum is taken over all rectifiable paths $\gamma$ in $\Omega$ that join the points $a,b$.  We describe this by calling  $\rho\,ds=\rho(z)|dz|$ a \emph{conformal metric} on $\Omega$.  When $(\Omega,d_\rho)$ is complete, the Hopf-Rinow Theorem (see \cite[p.35]{BH99}, \cite[p.51]{BBI01}, \cite[p.62]{Pap05}) asserts that $(\Omega,d)$ is a proper geodesic metric space.\footnote{%
See also \cite[Theorem~2.8]{Mar} for Euclidean domains.}

When $\rho$ is sufficiently differentiable,  say in $\mathcal{C}^2(\Omega)$, the \emph{Gaussian curvature} of $(\Omega,d_\rho)$ is given by \eqref{E:cvtr}
%%%\begin{equation}\label{E:cvtr}
%%%  {\mathbf K}_\rho:=-\rho^{-2}\Delta\log\rho\,.
%%%\end{equation}
This curvature can also be defined for continuous metric densities through an integral formula for the Laplacian, although it is not always finite;  see \cite[\S3]{MO86}. This idea was first observed by Heins \cite{Heins} who proved a version of Schwarz' lemma when the weak curvature is bounded above by $-1$.

Every hyperbolic plane domain $\Omega$
%\footnote{%So $\Omega\subsetneq\bbC$ has at least two finte boundary points.}
carries a unique metric $\lambda\,ds=\lambda_\Omega\,ds$ which enjoys the property that its pullback $\Phi^*[\lambda\,ds]$, with respect to any holomorphic universal covering projection $\Phi:\bbD\to\Omega$, is the hyperbolic metric $\lambda_\bbD(\zeta)|d\zeta|=2(1-|\zeta|^2)^{-1}|d\zeta|$ on $\bbD$.
%%%In terms of such a covering $\Phi$, the metric-density $\lambda=\lambda_\Omega$ of the {\em Poincar\'e hyperbolic metric\/} $\lambda_\Omega\,
%%%ds$ can be determined from
%%%\[
%%%  \lambda(z)=\lambda_\Omega(z)=\lambda_\Omega(p(\zeta))=2(1-|\zeta|^2)^{-1}|p'(\zeta)|^{-1}\,.
%\]
Alternatively, $\lambda\,ds$ is the unique maximal (or unique complete) metric on $\Omega$ that has constant Gaussian curvature $-1$.

%%%%---------------------------------------------------------------------------%
%%%\subsubsection{The Quasihyperbolic Metric \& Distance} 	\label{ss:qhyp m&d} %
%%%%---------------------------------------------------------------------------%
%%%The \emph{quasihyperbolic distance} $k=k_\Omega$ in $\Omega$ is the length distance $k_\Omega:=d_{\delta^{-1}}$ induced by the 

The \emph{quasihyperbolic metric} on $\Omega\subsetneq\bbC$ is $\delta^{-1}ds$, and it is well-known that this metric is complete.\footnote{%
This is also true for rectifiably connected non-complete locally complete metric spaces as discussed in \cite[2.4.1, p.43]{HRS20}.}  In \cite[Corollary 3.6, p.44]{MO86} Martin and Osgood proved that in plane domains the quasihyperbolic metric has non-positive generalized curvature.

\subsection{CAT(0) Metric Spaces}      \label{s:CAT0} %
%-----------------------------------------------------%
Here our terminology and notation conforms exactly with that in \cite{BH99} and we refer the reader to this delightful trove of geometric information about non-positive curvature, and also see \cite{BBI01}.  We recall a few fundamental concepts, mostly  copied directly from \cite{BH99}.  Throughout this subsection, $X$ is a geodesic metric space; for example, $X$ could be a quasihyperbolic plane domain with its quasihyperbolic distance, or a closed rectifiably connected plane set with its intrinsic length distance.

%%----------------------------------------------------------------------%%
\subsubsection{Geodesic and Comparison Triangles}  \label{ss:triangles} %%
%%----------------------------------------------------------------------%%
A \emph{geodesic triangle} $\Delta$ in $X$ consists of three points in $X$, say $a,b,c\in X$, called the \emph{vertices of $\Delta$} and three geodesics, say $\alpha:a\curvearrowright b, \beta:b\curvearrowright c, \gamma:c\curvearrowright a$ (that we may write as $[a,b], [b,c], [c,a]$) called the \emph{sides of $\Delta$}.  We use the notation
\[
  \Delta=\Delta(\alpha,\beta,\gamma) \quad\text{or}\quad \Delta=[a,b,c]:=[a,b]\star[b,c]\star[c,a] \quad\text{or}\quad \Delta=\Delta(a,b,c)
\]
depending on the context and the need for accuracy.

A Euclidean triangle $\bar\Delta=\Delta(\bar{a},\bar{b},\bar{c})$ in $\IC$ is a \emph{comparison triangle} for $\Delta=\Delta(a,b,c)$ provided $|a-b|=|\bar{a}-\bar{b}|,|b-c|=|\bar{b}-\bar{c}|,|c-a|=|\bar{c}-\bar{a}|$.  We also write $\bar\Delta=\bar\Delta(a,b,c)$ when a specific choice of $\bar{a},\bar{b},\bar{c}$ is not required.  A point $\bar{x}\in[\bar{a},\bar{b}]$ is a \emph{comparison point} for $x\in[a,b]$ when $|x-a|=|\bar{x}-\bar{a}|$.

%%-------------------------------------------------%%
\subsubsection{CAT(0) Definition}  \label{ss:CAT0} %%
%%-------------------------------------------------%%

A geodesic triangle $\Delta$ in $X$ satisfies the \emph{CAT(0) distance inequality} if and only if the distance between any two points of $\Delta$ is not larger than the Euclidean distance between the corresponding comparison points; that is,
\[
  \forall\; x,y\in\Delta\;\text{and corresponding comparison points}\;\bar{x},\bar{y}\in\bar\Delta\;, \quad |x-y| \le |\bar{x}-\bar{y}|\,.
\]
A geodesic metric space is \emph{CAT(0)} if and only if each of its geodesic triangles satisfies the CAT(0) distance inequality.

A complete CAT(0) metric space is called a \emph{Hadamard space}.  A geodesic metric space $X$ has \emph{non-positive curvature} if and only if it is locally CAT(0), meaning that for each point $a\in X$ there is an $r>0$ (that can depend on $a$) such that the metric ball $B(a;r)$ (endowed with the distance from $X$) is CAT(0).

Each sufficiently smooth Riemannian manifold has non-positive curvature if and only if all of its sectional curvatures are non-positive; see \cite[Theorem~1A.6, p.173]{BH99}.  In particular, if $\rho\,ds$ is a smooth conformal metric on $\Omega$ with ${\bf K}_\rho\le0$, then $(\Omega,d_\rho)$ has non-positive curvature.

%-------------------------------------------------------%
\subsection{Some Potential Theory}      \label{ss:potl} %
%-------------------------------------------------------%
Recalling \eqref{E:cvtr}, we see that when $\Omega\xrightarrow{\rho}(0,+\infty)$ is $\mathcal{C}^3$ smooth with $\log\rho$ subharmonic in $\Omega$, then $(\Omega,d_\rho)$ is a metric space of non-positive curvature; see \cite[Theorem~1A.6, p.173]{BH99}.  We utilize this basic fact, but also require the following.

%::::::::::::::::::::::::::::::%
\begin{lemma} \label{L:subhar} %
Suppose $\Omega\xrightarrow{\rho}(0,+\infty)$ is $\mathcal{C}^2$ smooth with $\log\rho$ subharmonic in $\Omega$.  Then $\log(1+\rho)$ is subharmonic in $\Omega$.
\end{lemma} %::::::::::::::::::%
%>>>>>>>>>>>>%
\begin{proof}%
%>>>>>>>>>>>>%
As $\rho=\exp(\log\rho)$ (with $\exp$ a convex function), it is also subharmonic in $\Omega$.  Thus
\begin{gather*}
  \rho_{z\zbar} \ge 0 \quad\text{and}\quad
  \rho^2 \frac{\partial^2}{\partial\zbar\partial z}\Bigl[\log\rho\Bigr] = \rho\rho_{z\zbar}-\rho_z\rho_{\zbar} \ge 0\,,
  \intertext{whence}
  \frac{\partial^2}{\partial\zbar\partial z}\biggl[\log(1+\rho)\biggr]= \frac{\partial}{\partial\zbar}\biggl[\frac{\rho_z}{1+\rho}\biggr] = \frac{(1+\rho)\rho_{z\zbar}-\rho_z\rho_{\zbar}}{(1+\rho)^2} \ge0\,.
\end{gather*}
%<<<<<<<<<<%
\end{proof}%
%<<<<<<<<<<%

\begin{rmks}\label{R:subhar}
We can apply the above:
\begin{enumerate}[\rm(a), wide, labelwidth=!, labelindent=20pt]
%  \item[ ]
  \item  to $\varepsilon\rho$ for any $\varepsilon>0$;
  \item  to $\rho(z):=|z|^\alpha$ in $\Co$ for any $\alpha\in\bbR$;
  \item  whenever $\rho\,ds$ (is sufficiently smooth and) has curvature $\mathbf{K}_\rho\le0$.
\end{enumerate}
\end{rmks}

%%%The following is well-known.
%%%%.............................%
%%%\begin{fact}  \label{F:radial}%
%%%%.............................%
%%%  Suppose that $u$ is harmonic and `radial' in an annulus $A:=\{r<|z-p|<R\}$, meaning that $u(z)$ depends only on $|z-p|$.  Then there are $\alpha,\beta\in\bbR$ such that for all $z\in A$, $u(z)=\alpha\log|z-p|+\beta$.
%%%%..........%
%%%\end{fact} %
%%%%..........%

%------------------------------------------------%
\subsubsection{Smoothing}      \label{ss:smooth} %
%------------------------------------------------%
Let $\IC\xrightarrow{\eta}\IR$ be $\mathcal{C}^\infty$ smooth with $\eta\ge0$, $\eta(z)=\eta(|z|)$, the support of $\eta$ in $\ID$, and $\int_\IC\eta=1$.  For each $\varepsilon >0$ we set $\eta_\varepsilon (z):=\varepsilon ^{-2}\eta(z/\varepsilon )$.  The \emph{regularization} (or \emph{mollification}) of an $L^{1}_{\rm loc}(\Omega)$ function $u:\Omega\to\IR$ are the convolutions $u_\varepsilon :=u\ast\eta_\varepsilon $, so
\[
  u_\varepsilon (z) := \int_\IC u(w)\eta_\varepsilon (z-w)\, dA(w)\,,
\]
which are defined in $\Omega_\varepsilon :=\{z\in\Omega:\delta(z)>\varepsilon \}$.  It is well known that $u_\varepsilon \in\mathcal{C}^\infty(\Omega_\varepsilon )$ and $u_\varepsilon \to u$ as $\varepsilon \to0^+$ where this convergence is: pointwise at each Lebesgue point of $u$, locally uniformly in $\Omega$ if $u$ is continuous in $\Omega$, and in $L^p_{\rm loc}(\Omega)$ if $u\in L^p_{\rm loc}(\Omega)$.  Moreover, if $u$ is subharmonic in $\Omega$, then so is each $u_\varepsilon $.  See for example \cite[Proposition I.15, p.235]{GL86} or \cite[Theorem~2.7.2, p.49]{Ran95}.

%%==========================================%%
\section{Proof of Theorem}  \label{S:PfThm} %%
%%==========================================%%
Let $\rho\,ds$ be a conformal metric on a plane domain $\Omega$ with an induced length distance $d:=d_\rho$ that is complete.  Suppose $\rho=\varphi\circ u$ where $u$ is subharmonic in $\Omega$, $\varphi$ is positive and increasing on an interval containing $u(\Omega)$, and $\log\varphi$ is convex.  We demonstrate that the metric universal cover
$(\tilde{\Omega},\tilde{d})$ is a Hadamard space.

The main idea is to approximate $(\Omega,d)$ by metric spaces $(\Omega_\varepsilon ,d_\varepsilon )$ that all have non-positive curvature.  Then a limit argument, similar to that used in the proof of \cite[Theorem~A]{Her21a}, gives the asserted conclusion.

We start with the fact that $v:=\log\rho=\log\varphi\circ u$ is subharmonic in $\Omega$.  See \cite[Theorem~2.6.3, p.43]{Ran95}.
%%%  This is easy to see.  Evidently, $u$ is continuous.  For each fixed $\zeta\in\partial\Omega$, $z\mapsto-\log|z-\zeta|$ is harmonic in $\IC\setminus\{\zeta\}\supset\Omega$ and so has the mean value property in $\Omega$.  It follows that $u$ has the submean value property in $\Omega$.
Let $v_\varepsilon :=v\ast\eta_\varepsilon $ be the regularization of $v$ as described in \S\ref{ss:smooth}.  Thus $v_\varepsilon$ are defined, $\mathcal{C}^\infty$ smooth, and subharmonic (so, $\Delta v_\varepsilon \geq 0$) in $\{z\in\Omega:\delta(z)>\varepsilon \}$.  Moreover,  $v_\varepsilon \to v$ as $\varepsilon \to0^+$ locally uniformly in $\Omega$.

\medskip

The elementary cases where $\Omega=\bbC$ or $\Omega$ is the once punctured plane are left to the reader.    

\medskip

We may therefore assume that $\Omega$ is a hyperbolic plane domain and that the origin lies in $\Omega$.  Put $\varepsilon _n:=\delta(0)/n, v_n:=v_{\varepsilon _n}$, and let $\Omega_n$ be the component of $\{z\in\Omega:\delta(z)>\varepsilon _n\}$ that contains the origin.  Then $(\Omega_n)_{n=1}^{\infty}$ Carath\'eodory kernel converges to $\Omega$  with respect to the origin.

Next, let $\rho_n:=e^{v_n}$.  Then $\rho_n>0$ and is $\mathcal{C}^\infty$ in $\Omega_n$.  Since $v_n$ is subharmonic in $\Omega_n$, $\rho_n\,ds$ has Gaussian curvature
\[
{\bf K}_{\rho_n}=-\rho_n^{-2}\Delta\log\rho_n\le0\quad\text{in $\Omega_n$}\,.
\]
It follows that the metric spaces $(\Omega_n,d_{\rho_n})$ all have non-positive curvature; see \cite[Theorem~1A.6, p.173]{BH99}.

We would like to appeal to the Cartan-Hadamard theorem to assert that the metric universal coverings of $(\Omega_n,d_{\rho_n})$ are CAT(0), but these metric spaces need not be complete.\footnote{%
We are grateful to the referee for pointing out this glaring gap  in our original argument.}
To overcome this roadblock, we employ the hyperbolic metrics $\lambda_n\,ds$ in the domains $\Omega_n$.  We define metrics $\sigma_n\,ds$ on $\Omega_n$ via
\[
  \forall\; z\in\Omega\,,\quad \sigma_n(z):=\rho_n(z)\cdot\bigl(1+\varepsilon_n\lambda_n(z)\bigr)\,.
\]
According to Fact~\ref{L:subhar} (using Remarks~\ref{R:subhar}(a,c)) $\log\sigma_n$ is subharmonic in $\Omega$.  Thus $\sigma_n\,ds$ induces a distance $d_n:=d_{\sigma_n}$ with $(\Omega_n,d_n)$ a complete geodesic metric space of non-positive curvature.  We note that $\sigma_n\to\rho$ locally uniformly in $\Omega$.

Let $\bbD\xrightarrow{\Phi}\Omega, \bbD\xrightarrow{\Phi_n}\Omega_n$ be holomorphic covering projections with $\Phi(0)=0=\Phi_n(0)$ and $\Phi'(0)>0, \Phi'_n(0)>0$.  Since $(\Omega_n)$ Carath\'eodory kernel converges to $\Omega$ with respect to the origin, a theorem of Hejhal's \cite[Theorem~1]{Hej74}\footnote{%
See also \cite[Cor.~5.3]{BM22}.}
asserts that $\Phi_n\to \Phi$, so also $\Phi'_n\to \Phi'$, locally uniformly in $\bbD$.

Let $\tilde{d}, \tilde{d}_n$ be the $\Phi,\Phi_n$ lifts of the distances $d,d_n$ on $\Omega, \Omega_n$ respectively.  That is, $\tilde{d}$ and $\tilde{d}_n$ are the length distances on $\bbD$ induced by the pull backs
\begin{gather*}
  \tilde\rho\,ds:=\Phi^*\bigl[\rho\,ds] \quad\text{and}\quad  \tilde\sigma_n\,ds:=\Phi_n^*\bigl[\sigma_n\,ds]
  \intertext{of the metrics $\rho\,ds$ and $\sigma_n\,ds$ in $\Omega$ and $\Omega_n$ respectively.  Thus, for $\zeta\in\bbD$,}
  \tilde\rho(\zeta)\,|d\zeta|=\rho\bigl(\Phi(\zeta)\bigr)|\Phi'(\zeta)|\,|d\zeta|  \quad\text{and}\quad  \tilde\sigma_n(\zeta)\,|d\zeta|= \sigma_n\bigl(\Phi_n(\zeta)\bigr)|\Phi_n'(\zeta)| \,|d\zeta|
\end{gather*}
and $(\bbD,\tilde{d})\xrightarrow{\Phi}(\Omega,d), (\bbD,\tilde{d}_n)\xrightarrow{\Phi_n}(\Omega_n,d_n)$ are metric universal coverings.

Note that as $(\Omega_n,d_n)$ has non-positive curvature, the Cartan-Hadamard Theorem \cite[Chapter~II.4, Theorem~4.1, p.193]{BH99} asserts that $(\bbD,\tilde{d}_n)$ is CAT(0).

Using the locally uniform convergences of $\sigma_n\to\rho$ and $\Phi_n\to\Phi, \Phi'_n\to\Phi'$ (in $\Omega$ and $\bbD$ respectively) we deduce that $\tilde\sigma_n\,ds\to\tilde\rho\,ds$ locally uniformly in $\bbD$. This implies pointed Gromov-Hausdorff convergence of $(\bbD,\tilde{d}_n,0)$ to $(\bbD,\tilde{k},0)$ (see the proof of \cite[Theorem~4.4]{HRS20}) which in turn says that $(\bbD,\tilde{k})$ is a 4-point limit of $(\bbD,\tilde{d}_n)$ and hence, as each $(\bbD,\tilde{d}_n)$ is CAT(0), it follows that $(\bbD,\tilde{d})$ is CAT(0); see \cite[Cor.~3.10, p.187; Theorem~3.9, p.186]{BH99}.  Finally, it is a routine matter to check that $(\bbD,\tilde{d})$ is complete; for instance, see \cite[Exercise~3.4.8, p.~80]{BBI01}.
\hfill \qed

\medskip
Finally,  the smoothness of geodesics in case the metric density is locally Lipschitz is proved in \cite[Theorem 2.12 \& Theorem 4.3]{Mar}.

%%%DAH,  Department of Mathematical Sciences, University of Cincinnati, OH 45221-0025, USA \\
%%%email: David.Herron@UC.edu
%%%
%%%\medskip
%%%
%%%GJM, Institute for Advanced Study, Massey University, Auckland, New Zealand \\
%%%email: G.J.Martin@Massey.ac.NZ

\end{document}